\documentclass[11pt,notitlepage,a4paper]{article}

\usepackage{amssymb} 
\usepackage[intlimits]{amsmath}
\usepackage{amsfonts}
\usepackage{amsthm} 

\usepackage{mathptmx}

\usepackage{hyperref}
\hypersetup{colorlinks=true,linkcolor=RoyalPurple,citecolor=RoyalPurple}

\usepackage{cite}

\usepackage{graphicx}

\usepackage{geometry}

\usepackage{color}
\usepackage[dvipsnames]{xcolor}

\let\oldsqrt\sqrt
\def\sqrt{\mathpalette\DHLhksqrt}
\def\DHLhksqrt#1#2{%
\setbox0=\hbox{$#1\oldsqrt{#2\,}$}\dimen0=\ht0
\advance\dimen0-0.2\ht0
\setbox2=\hbox{\vrule height\ht0 depth -\dimen0}%
{\box0\lower0.4pt\box2}}

\allowdisplaybreaks

\newcommand{\R}{\mathbb{R}} 
\newcommand{\N}{\mathbb{N}} 

\newcommand{\supp}{\textnormal{supp}} 
\newcommand{\esssup}{\textnormal{esssup}} 

\renewcommand{\phi}{\varphi}

\newcommand{\cD}{{\mathcal D}}
\newcommand{\cE}{{\mathcal E}}

\newcommand{\cL}{{\mathcal L}}

\newcommand{\weakto}{\rightharpoonup}
\newcommand{\weak}{\rightharpoonup}

\newcommand{\eps}{\varepsilon}

\theoremstyle{definition}

\theoremstyle{plain} 
\newtheorem{defi}{Definition}[section]
\newtheorem{thm}[defi]{Theorem}
\newtheorem{prop}[defi]{Proposition}
\newtheorem{lemma}[defi]{Lemma}
\newtheorem{cor}[defi]{Corollary}
\newtheorem{remark}[defi]{Remark}

\theoremstyle{definition}

\numberwithin{equation}{section} 

\title{Local compactness and nonvanishing for weakly singular nonlocal quadratic forms}
\author{
 \ Sven Jarohs\footnote{Institut f\"ur Mathematik, Goethe-Universit\"at, Frankfurt, Robert-Mayer-Stra\ss e 10, D-60629 Frankfurt, jarohs@math.uni-frankfurt.de.}
 \;   and
 \!\ Tobias Weth\footnote{Institut f\"ur Mathematik, Goethe-Universit\"at, Frankfurt, Robert-Mayer-Stra\ss e 10, D-60629 Frankfurt, weth@math.uni-frankfurt.de.}}
\date{\today}

\begin{document}
\maketitle

\begin{abstract}
In this work we study a class of nonlocal quadratic forms given by 
\[
\cE_j(u,v)=\frac{1}{2}\int_{\R^N}\int_{\R^N}(u(x)-u(y))(v(x)-v(y))j(x-y)\ dxdy,
\]
where $j:\R^N\to[0,\infty]$ is a measurable even function with $\min\{1,|\cdot|^2\}j\in L^1(\R^N)$. Assuming merely $j\notin L^1(\R^N)$, we show local compactness of the embedding $\cD^j(\R^N)\hookrightarrow L^2(\R^N)$, where $\cD^j(\R^N)$ denotes the space of functions $u\in L^2(\R^N)$ with $\cE_j(u,u)<\infty$. Using this local compactness, we establish an alternative which allows to distinguish vanishing and nonvanishing of bounded sequences in $\cD^j(\R^N)$. As an application, we show the existence of maximizers for a class of integral functionals defined on the unit sphere in $\cD^j(\R^N)$. Our main results extend to cylindrical unbounded sets of the type $\Omega = U \times \R^k$, where $U \subset \R^{N-k}$ is open and bounded. Finally, we note that a Poincar\'e inequality associated with $\cE_j$ holds for unbounded domains of this type, thereby extending the corresponding result in \cite{FKV13} for bounded domains.
\end{abstract}
{\footnotesize
\begin{center}
\end{center}
\begin{center}
%
\end{center}
}

\section{Introduction}\label{sec:introduction}

The present paper is devoted to quadratic forms and function spaces associated with unbounded nonlocal operators $I$ on $L^2(\R^N)$ formally given by 
\begin{equation}
  \label{eq:def-Iu-formal}
Iu(x)=P.V.\int_{\R^N}(u(x)-u(y))j(x-y)\ dy : = \lim_{\eps \to 0^+} 
\int_{\R \setminus B_{\eps}(x)}(u(x)-u(y))j(x-y)\ dy.
\end{equation}
Here $j: \R^N \to [0,\infty]$ is the associated (nonnegative) kernel function, which typically has a singularity at the origin. In recent years, operators of this type have received increasing attention, whereas most of the work has been devoted to the case $j(z)=|z|^{-N-\alpha}$ with $\alpha \in (0,2)$. In this case, $I$ equals, up to a multiple constant, the fractional Laplacian of order $\alpha$, see e.g. \cite{BV15} and the references therein.\\
In the present paper, we wish to derive some useful tools for the study of quadratic forms associated to a very general class of operators of type (\ref{eq:def-Iu-formal}) without any restriction of the order. More precisely, we assume that $j: \R^N \to [0,\infty]$ satisfies 
\begin{equation*}\label{j1}
(A1) \qquad j(z)=j(-z)\quad \text{ for all $z\in \R^N\ $ and}\quad  0< \int_{\R^N}\min\{1,|z|^2\}j(z)\ dz<\infty.
\end{equation*}
If $(A1)$ holds, then $Iu \in C(\R^N)$ is well-defined on $\R^N$ by (\ref{eq:def-Iu-formal}) for 
functions $u \in C^2_c(\R^N)$. Moreover, we have 
$$
\langle Iu,v \rangle = \cE_j(u,v) \qquad \text{for all $u,v \in C^2_c(\R^N)$} 
$$
with the associated bilinear form 
\begin{equation}
\label{bilinear}
(u,v) \mapsto \cE_j(u,v)=\frac{1}{2}\int_{\R^N}\int_{\R^N} (u(x)-u(y))(v(x)-v(y))j(x-y)\ dxdy.
\end{equation}
Note that $\cE_j$ is well-defined on the space 
\begin{equation}\label{D-j-def-r-n}
\cD^j(\R^N):=\{u\in L^2(\R^N)\;:\; \int_{\R^N}\int_{\R^N} (u(x)-u(y))^2j(x-y)\ dxdy <\infty\}.
\end{equation}
More generally, for an open set $\Omega \subset \R^N$, we define 
\begin{equation}\label{D-1-def-omega}
\cD^j(\Omega):=\{u\in \cD^j(\R^N)\;:\;  u \equiv 0 \;\text{on $\R^N \setminus \Omega$}\}.
\end{equation}
It is known and not difficult to see that $\cD^j(\Omega)$ is a Hilbert space with scalar product $\langle \cdot,\cdot \rangle$ given by 
$$
\langle u,v \rangle = \cE_j(u,v) + \langle u,v \rangle_{L^2(\R^N)}, 
$$
and corresponding norm $\|\cdot\|$, see \cite{JW14,FKV13}. Moreover, $\cD^j(\Omega) \subset L^2(\Omega)$ is dense, since it contains the space of $C^1$-functions with compact support in $\Omega$. Here and in the following, we identify $L^2(\Omega)$ with the space of functions in $L^2(\R^N)$ with $u \equiv 0$ on $\R^N \setminus \Omega$. It thus follows that $\cE_j$ is the quadratic form of a unique self-adjoint operator $I$ in $L^2(\Omega)$. Moreover, $C_c^2(\Omega)$ is contained in the domain of $I$, and for every $v \in C_c^2(\Omega)$ the function $Iv \in  L^2(\Omega)$ is a.e. given by (\ref{eq:def-Iu-formal}). For proofs of these statements, see e.g. \cite[Section 2]{JW14}.

The first main purpose of the present paper is to study (local) compactness properties of the embedding $\cD^j(\R^N) \hookrightarrow L^2(\R^N)$. In the following, for a measurable subset $K \subset \R^N$, we let $1_K$ denote the characteristic function of $K$ and 
$$
R_K: L^2(\R^N) \to L^2(\R^N),\qquad R_K u = 1_k \, u
$$
the corresponding multiplication operator with $1_K$. Moreover, if $E$ is a normed vector space, we call a continuous linear operator $T: E \to L^2(\R^N)$ {\em locally compact} if $R_K T: E \to L^2(\R^N)$ is a compact operator for every compact subset $K \subset \R^N$. 

We note that, if the embedding $\cD^j(\R^N) \hookrightarrow L^2(\R^N)$ is locally compact, then the embedding $\cD^j(\Omega) \hookrightarrow L^2(\Omega)$ is compact for every bounded open set $\Omega \subset \R^N$. A necessary condition for the local compactness of the embedding $\cD^j(\R^N) \hookrightarrow L^2(\R^N)$ is the following:
\begin{equation*}
(A2) \qquad \int_{\R^N}j(z)\ dz=\infty.
\end{equation*}
Indeed, if on the contrary $j \in L^1(\R^N)$, then the spaces $\cD^j(\R^N)$ und $L^2(\R^N)$ coincide
 with equivalent norms, since 
 $$
\|w\|_{L^2(\R^N)}\leq \|w \|\leq (1+2\|j\|_{L^1(\R^N)})^{1/2} \|w\|_{L^2(\R^N)}
\qquad \text{for every $w \in L^2(\R^N)$.}
$$
Consequently, local compactness fails in this case. Assumption $(A2)$ should thus be regarded as the weakest possible singularity condition on the kernel under which local compactness might be expected. In recent years, there has been an increasing interest in nonlocal operators with weakly singular kernels as they correspond to non-fractional orders near zero, see e.g. \cite{JW17,KM17,CP17,CW18} and the references therein.

In our first main result, we shall see that $(A2)$ is indeed also a sufficient condition for local compactness.

\begin{thm}
\label{sec:local-compactness-2}
Suppose that $j$ satisfies $(A1)$ and $(A2)$. Then the embedding $\cD^j(\R^N) \hookrightarrow L^2(\R^N)$ is locally compact.
\end{thm}

As noted already, Theorem~\ref{sec:local-compactness-2} implies that the embedding $\cD^j(\Omega) \hookrightarrow L^2(\Omega)$ is compact for every bounded open set $\Omega \subset \R^N$. The latter result has been shown in \cite[Theorem 2.1]{CP17} under the assumption that $j$ is a {\em radially symmetric} kernel satisfying assumptions $(A1)$, $(A2)$ and such that $j$ is positive and slowly varying a neighborhood of the origin (see conditions (H1) and (H2) in \cite{CP17}). These additional restrictions are used in the proof in \cite{CP17} which is based on pointwise estimates for the Fourier symbol of the operator $I$ as derived in \cite[Proposition 6]{KM17}.

To deal with nonradial kernels and without additional assumptions on $j$, 
we apply a completely different and surprisingly simple argument based on weighted averages, where a cut-off of the kernel $j$ is used as a weight function. This also provides an alternative simple proof in the classical case where $j(z)=|z|^{-N-\alpha}$ with $\alpha \in (0,2)$, which corresponds to the fractional Laplacian. 

Our next result extends the compactness statement for the embedding $\cD^j(\Omega) \hookrightarrow L^2(\Omega)$ to (possibly unbounded) sets of finite measure.

\begin{thm}
  \label{thm-cor-sec:local-compactness-5}
Suppose that $j$ satisfies $(A1)$ and $(A2)$, and let $\Omega \subset \R^N$ be an open set with $|\Omega|<\infty$. Then the embedding $\cD^j(\Omega) \hookrightarrow L^2(\Omega)$ is compact. 
\end{thm}

Theorem~\ref{thm-cor-sec:local-compactness-5} will be deduced from Theorem~\ref{sec:local-compactness-2} and from additional estimates for the killing measure associated with $j$ and for projections on subsets of $L^\infty$-bounded functions.
Under the assumptions of Theorem~\ref{thm-cor-sec:local-compactness-5}, it follows in a standard way that the associated selfadjoint operator $I$ in $L^2(\Omega)$ defined above admits a sequence of eigenvalues
$$
0 < \lambda_1(\Omega)\le \lambda_2(\Omega)\le \cdots\le \lambda_k(\Omega)\le \lambda_{k+1}(\Omega)\le \cdots
$$
with $\lim \limits_{k \to \infty}\lambda_k(\Omega) = \infty$ and an orthonormal basis of $\cD^j(\Omega)$ of associated eigenfunctions $\xi_k$, $k \in \N$.

Compactness of the embedding $\cD^j(\Omega) \hookrightarrow L^2(\Omega)$ 
fails in general if $\Omega$ has infinite measure. An obvious indication for this failure is the translation invariance of the quadratic form $\cE_j$. Our next theorem distinguishes vanishing and nonvanishing properties of bounded sequences with respect to translations. In the following, for $u \in \cD^j(\R^N)$ and $x \in \R^N$, we define 
 $x* u \in \cD^j(\R^N)$ by $[x*u](y)=u(y-x)$.

\begin{thm}
\label{sec:vanishing-nonvanishing}
Suppose that $j$ satisfies $(A1)$ and $(A2)$, and let $(u_n)_n \subset \cD^j(\R^N)$ be a bounded sequence. Then one of the following alternatives holds.
\begin{enumerate}
\item[(i)] $$
\lim \limits_{n \to \infty} \int_{|u_n| \ge \eps}u_n^2\,dx = 0\qquad \text{for every $\eps>0$.}
$$
\item[(ii)] There exists a sequence $(x_n)_n \subset \R^N$ such that, after passing to a subsequence, we have  
$$
x_n * u_n \weakto u \not = 0 \qquad \text{in $\cD^j(\R^N)$.}
$$
\end{enumerate}
\end{thm}

Somewhat similarly as Theorem~\ref{thm-cor-sec:local-compactness-5}, Theorem~\ref{sec:vanishing-nonvanishing} will also be deduced from Theorem~\ref{sec:local-compactness-2} and from estimates related to the killing measure associated with $j$.
Theorem~\ref{sec:vanishing-nonvanishing} should be compared with a classical result of Lions which states that a bounded sequence $(u_n)_n$ in the Sobolev space $H^1(\R^N)$ converges to zero in $L^p(\R^N)$ for $2<p<2^*$  if it satisfies 
\begin{equation}
  \label{eq:lions-vanishing}
\lim_{n \to \infty} \,\sup_{y \in \R^N} \int_{B_r(y)}|u_n|^2\,dx = 0 \qquad \text{for some $r>0$}
\end{equation}
(see \cite[Lemma I.1]{L84} or \cite[Lemma 1.21]{Willem}). Here $2^*$ denotes the first order critical Sobolev exponent, i.e., $2^* =\frac{2N}{N-2}$ for $N \ge 3$ and $2^*= \infty$ for $N=1,2$. 
The proof of Lions' result strongly relies on the embedding of $H^1(\R^N)$ into $L^p(\R^N)$ for $2 \le p < 2^*$. Consequently, the argument does not extend to the setting of Theorem~\ref{sec:vanishing-nonvanishing} since -- under the present assumptions -- the space $\cD^j(\R^N)$ might not embed into any space $L^p(\R^N)$ with $p>2$.

Theorem~\ref{sec:vanishing-nonvanishing} is useful in the study of certain classes of variational problems. As an application, we consider the maximization problem associated with 
\begin{equation}
  \label{eq:def-m-F-intro}
m_{F,\R^N}:= \sup_{u \in S } \Phi(u),
\end{equation}
where $S(\R^N):= \{u \in \cD^j(\R^N) \::\: \|u\| =1\}$ is the unit sphere in $\cD^j(\R^N)$ and 
$$
\Phi: L^2(\R^N) \to \R, \qquad \Phi(u)= \int_{\R^N}F(u)\,dx
$$
is an integral functional associated with a continuous function $F:\R \to \R$ satisfying the following assumptions. 
\begin{itemize}
\item[(F1)] $F(t) \le {c_\infty \,} t^2$ for every $t \in \R$ with some constant ${c_\infty \,}>0$; 
\item[(F2)] $F(t)=o(t^2)$ as $t \to 0$; 
\item[(F3)] $\frac{F(t)}{t^2}$ is nondecreasing on $[0,\infty)$ and nonincreasing on $(-\infty,0]$. 
\end{itemize}
Typical examples for functions $F$ satisfying these assumptions are $t \mapsto \frac{t^4}{1+t^2}$ or $t \mapsto  \int_0^t \frac{s^3}{1+s^2}\,ds$. Assumptions (F2) and (F3) yield that $F$ is nonnegative. Assumption (F1) implies that $\Phi$ is bounded on bounded subsets of $L^2(\R^N)$. Since $F$ is moreover continuous, a classical argument shows that $\Phi$ is continuous on $L^2(\R^N)$. It also follows from (F1) that $m_{F,\R^N}<\infty$. Moreover, $m_{F,\R^N}>0$ if and only if $F \not \equiv 0$.  We have the following result.
\begin{thm}
\label{max-value-attained}
Assume that $j$ satisfies $(A1)$ and $(A2)$, and that $F\in C(\R)$ satisfies (F1)--(F3). Then the maximal value $m_{F,\R^N}$ is attained, i.e., there exists $u \in S(\R^N)$ with $\Phi(u)=m_{F,\R^N}$.
\end{thm}

Both Theorem~\ref{sec:vanishing-nonvanishing} and Theorem~\ref{max-value-attained} admit straightforward extensions to cylindrical open sets of the type $\Omega:= U \times \R^k$, where $N= n+k$ with $1 \le k,n \le N-1$ and $U \subset \R^n$ is a bounded open set. More precisely, we have the following.  

\begin{thm}
\label{sec:vanishing-nonvanishing-cylinder}
Let $\Omega:= U \times \R^k$, where $N= n+k$ and $U \subset \R^n$ is a bounded open set, assume that $j$ satisfies $(A1)$ and $(A2)$, and let $(u_n)_n \subset \cD^j(\Omega)$ be a bounded sequence. Then one of the following alternatives holds.
\begin{enumerate}
\item[(i)] $$
\lim \limits_{n \to \infty} \int_{|u_n| \ge \eps}u_n^2\,dx = 0\qquad \text{for every $\eps>0$.}
$$ 
\item[(ii)] There exists a sequence $(x_n)_n \subset \{0\} \times \R^k$ such that, after passing to a subsequence, we have  
$$
x_n * u_n \weakto u \not = 0 \qquad \text{in $\cD^j(\Omega)$.}
$$
\end{enumerate}
\end{thm}

\begin{thm}
\label{max-value-attained-cylinder}
Assume that $j$ satisfies $(A1)$ and $(A2)$. Let $\Omega:= U \times \R^k$, where $N= n+k$ and $U \subset \R^n$ is a bounded open set, let $S(\Omega):= \{u \in \cD^j(\Omega) \::\: \|u\| =1\}$, and let $F\in C(\R)$ satisfy (F1)--(F3). Then the maximal value 
$$
m_{F,\Omega}:= \sup_{u \in S(\Omega)} \Phi(u).
$$
of $\Phi$ on $S(\Omega)$ is attained, i.e., there exists $u \in S(\Omega)$ with $\Phi(u)=m_{F,\Omega}$.
\end{thm}

It is natural to ask whether Theorem~\ref{max-value-attained-cylinder} still holds if the set $S(\Omega)$ is replaced by 
$$
\tilde S(\Omega):= \{u \in \cD^j(\Omega)\::\: \cE_j(u,u)=1\}.
$$
This leads to the question whether $u \mapsto \sqrt{\cE_j(u,u)}$ defines a 
norm on $\cD^j(\Omega)$ which is equivalent to $\|\cdot\|$. In the case of bounded open sets $\Omega \subset \R^N$, this is indeed the case due to the Poincar\'e inequality given in \cite[Lemma 2.7]{FKV13}. Here we note that, by a simple variant of the argument given in \cite{FKV13}, the Poincar\'e inequality for $\cE_j$ extends to domains which are only bounded in one space direction. More precisely, we have the following. 

\begin{prop}[Poincar\'e Inequality]\label{poincare}
	Assume that $j$ satisfies (A1). Then for any $a>0$ there is a constant $C_a>0$ such that for any $\Omega\subset (-a,a)\times \R^{N-1}$ we have
        \begin{equation}
          \label{eq:poincare-first-claim}
	\cE_j(u,u)\geq C_a\int_{\Omega} u^2(x)\ dx.
        \end{equation}
	Moreover, $\liminf \limits_{a \to 0^+}\, C_a \ge  \int_{\R^N}j(z)\ dz$.
\end{prop}
We point out that $(A2)$ is not needed here, so Proposition~\ref{poincare} also applies to quadratic forms of convolution type with $j \in L^1(\R^N)$. We note that
Proposition~\ref{poincare} parallels the classical Poincar\'e inequality  for domains $\Omega\subset (-a,a)\times \R^{N-1}$,  which states that 
$$
\||\nabla u|\|_{L^2(\Omega)} \ge \frac{\|u\|_{L^2(\Omega)}}{a}\qquad \text{for all
$u \in H^1_0(\Omega)$,}
$$
where $H^1_0(\Omega)$ is a usual (homogenous) first order Sobolev space. 
This classical Poincar\'e inequality is a fundamental tool in the theory of weak solutions of second order elliptic boundary value problems. In particular, it gives a lower bound for the spectrum of the Dirichlet Laplacian in $\Omega$, and it leads to a maximum principle on narrow domains which then can be used in moving plane type arguments. Similar applications arise from Proposition~\ref{poincare}, and we leave them for future work.

The paper is organized as follows. In Section \ref{sec:local-compactness} we give the proof of Theorem \ref{sec:local-compactness-2}. In Section \ref{sec:estim-kill-meas}, we first derive a key estimate for the killing measure associated to $j$ , and we then complete the proofs of Theorem~\ref{thm-cor-sec:local-compactness-5} and Theorem~\ref{sec:vanishing-nonvanishing}. 
In Section~\ref{sec:appl-maxim-probl}, we focus on the maximization problem related to (\ref{eq:def-m-F-intro}), and we give the proof of Theorem \ref{max-value-attained}. Since the proofs of Theorems~\ref{sec:vanishing-nonvanishing-cylinder} and Theorem~\ref{max-value-attained-cylinder} are completely parallel to the ones of Theorems~\ref{sec:vanishing-nonvanishing} and Theorem~\ref{max-value-attained}, we skip them. Finally, Section \ref{sec:poincare-inequality} is devoted to the proof of Proposition \ref{poincare}.

Throughout the paper, we let $B_r:=B_r(0)$ denote the open ball of radius $r>0$ centered at zero.

\section{Local compactness}
\label{sec:local-compactness}

This section is devoted to the proof of Theorem \ref{sec:local-compactness-2}. 
Throughout this section, $j: \R^N \to [0,\infty]$ denotes a kernel function satisfying $(A1)$ and $(A2)$. We need the following lemma which we believe is known. Since we could not find the statement in the literature in this form, we give a simple proof for the convenience of the reader. 

\begin{lemma}
\label{sec:local-compactness-3}
Let $w \in L^1(\R^N)$. Then the corresponding convolution operator 
\begin{equation}
  \label{eq:def-conv-operator}
T_w: L^2(\R^N) \to L^2(\R^N),\qquad T_wu = w * u
\end{equation}
is locally compact.  
\end{lemma}

\begin{proof}
By Young's convolution inequality, $T_w$ is a continuous linear map. Let $K \subset \R^N$ be a compact subset.
We first consider the case where $w \in C^\infty_c(\R^N)$. Let $M \subset L^2(\R^N)$ be a bounded set. We show that  
$T_w(M) \subset L^2(\R^N)$ is an equicontinuous set of functions on $\R^N$. Indeed, if $u \in M$ and $x, h \in \R^N$, we have 
\begin{align*}
&|[T_wu](x+h)-[T_wu](x)| = \int_{\R^N}u(z)[w(x+h-z)-w(x-z)]\,dz \\
&\le \|u\|_{L^2(\R^N)} 
\Bigl(\int_{\R^N} |w(x+h-z)-w(x-z)|^2\,dz\Bigr)^{\frac{1}{2}}= \|u\|_{L^2(\R^N)}^2 \sqrt{d(h)}, 
\end{align*}
where 
$$
d(h):= \int_{\R^N} |w(y+h)-w(y)|^2\,dy \to 0 \qquad \text{as $h \to 0$}
$$
Hence $T_w(M)$ is equicontinuous, and therefore $[T_w](M)$ is a relatively compact subset of $C(K)$ by 
the Arzela-Ascoli Theorem (here we identify a function $u$ on $\R^N$ with its restriction to $K$). Consequently, $[R_K T_w](M)$ is relatively compact in $L^2(\R^N)$, and therefore 
$$
\text{$R_K T_w \in \cL(L^2(\R^N), L^2(\R^N))$ is a compact operator.}
$$
Next, for general $w \in L^1(\R^N)$, we let $(w_n)_n$ be a sequence in $C^\infty_c(\R^N)$ with $w_n \to w$ in $L^1(\R^N)$. Then we have 
$$
\|[R_K T_{w_n}-R_K T_w](u)\|_{L^2(\R^N)} \le \|(w_n-w)* u\|_{L^2(\R^N)} \le \|w_n-w\|_{L^1(\R^N)} \|u\|_{L^2(\R^N)} \quad \text{for $u \in L^2(\R^N)$}
$$
and therefore $R_K T_{w_n} \to R_K T_w$ in $\cL(L^2(\R^N),L^2(\R^N))$ as $n \to \infty$. Since we have already seen that $R_K T_{w_n}$ is compact for every $n \in \N$, the operator $R_K T_w$ is compact as well, as claimed. 
\end{proof}

In the next lemma, we estimate the $L^2$-distance of functions in $\cD^j(\R^N)$ to their weighted averages, where a cut-off of the kernel $j$ is used as a weight function.

\begin{lemma}
\label{sec:local-compactness-j-delta}
Let $\delta>0$ be such that $j_\delta:= j\, 1_{\R^N \setminus B_\delta} \in L^1(\R^N) \setminus \{0\}$, and consider the function $w_\delta:= \frac{j_\delta}{\|j_\delta\|_{L^1(\R^N)}}$. Then we have  
$$
\|u-T_{w_\delta} u\|_{L^2(\R^N)} \le \Bigl( \frac{2}{\|j_\delta\|_{L^1(\R^N)}}\Bigr)^{\frac{1}{2}}\|u\| \qquad \text{for $u \in \cD^j(\R^N)$,}
$$
where $T_{w_\delta}$ denotes the convolution operator with $w_\delta$ as defined in (\ref{eq:def-conv-operator}). 
\end{lemma}

\begin{proof}
Let $u \in \cD^j(\R^N)$. Then we have, by the evenness of $w_\delta$,  
$$
[T_{w_\delta} u](x) = [w_\delta * u](x)= \int_{\R^N}w_\delta(x-y)u(y)\,dy= \int_{\R^N}u(x+z)w_\delta(z)\,dz \qquad \text{for $x \in \R^N$.}
$$
Moreover, since $\|w_\delta\|_{L^1(\R^N)}=1$, Jensen's inequality implies that
\begin{align*}
\|u-T_{w_\delta} u\|_{L^2(\R^N)}^2 &= \int_{\R^N}\Bigl( u(x)- [T_{w_\delta} u](x)\Bigr)^2\,dx \\
&= \int_{\R^N} \Bigl( \int_{\R^N}[u(x)- u(x+z)]w_\delta(z)\,dz\Bigr)^2\,dx\le  \int_{\R^N} \int_{\R^N}[u(x)- u(x+z)]^2 w_\delta(z)\,dz \,dx\\
& \le  \frac{1}{\|j_\delta \|_{L^1(\R^N)}} \int_{\R^N} \int_{\R^N}[u(x)- u(x+z)]^2 j(z)\,dz \,dx \leq  \frac{2\|u\|^2}{\|j_\delta\|_{L^1(\R^N)}},
\end{align*}
as claimed.
\end{proof}

We now have all the tools to complete the

\begin{proof}[Proof of Theorem~\ref{sec:local-compactness-2}]
Let $M \subset \cD^j(\R^N)$ be a bounded set, and let $K \subset \R^N$ be compact. We need to show that $R_K(M) \subset L^2(\R^N)$ is relatively compact. Let $C:= \sup \limits_{u \in M}\|u\|$, and let $\eps>0$. Since $j$ satisfies $(A1)$ and $\int_{\R^N}j(z)\,dz = \infty$, there exists $\delta>0$ such that $\|j_\delta\|_{L^1(\R^N)} \ge \frac{2C^2}{\eps^2}$. Here and in the following, we use the notation from Lemma~\ref{sec:local-compactness-j-delta}. Moreover, $\tilde M:= [R_K T_{w_\delta}](M)$ is relatively compact in $L^2(\R^N)$ by Lemma~\ref{sec:local-compactness-3}. For $u \in M$, Lemma~\ref{sec:local-compactness-j-delta} implies that 
$$
\|R_K u - [R_K T_{w_\delta}]u\|_{L^2(\R^N)} \le \|u - T_{w_\delta} u\|_{L^2(\R^N)}\le 
\Bigl( \frac{2}{\|j_\delta\|_{L^1(\R^N)}}\Bigr)^{\frac{1}{2}} \|u\|\le \frac{\eps \|u\| }{C}\le \eps, 
$$
and therefore $R_K(M)$ is contained in the $\eps$-neighborhood of $\tilde M$. Since $\eps>0$ was chosen arbitrarily, we conclude that 
$R_K(M)$ is totally bounded in $L^2(K)$ and therefore relatively compact. The proof is finished.
\end{proof}

\section{An estimate for the killing measure and its consequences}
\label{sec:estim-kill-meas}
 In this section, we first derive estimates for the killing measure associated with the kernel $j$ in terms of its decreasing rearrangement. We then use these estimates to complete the proofs of Theorem~\ref{thm-cor-sec:local-compactness-5} and \ref{sec:vanishing-nonvanishing}. Throughout this section, we assume that 
$j: \R^N \to [0,\infty]$ satisfies $(A1)$. The decreasing rearrangement of $j$ is defined as 
$$
d_j:[0,\infty) \to [0,\infty], \qquad d_j(r)=  \inf \{c \ge 0 \::\: |\{j > c\}| \le r \}
$$
By definition, $d_j$ is a nonincreasing function, and it also has the following properties:
\begin{equation}
  \label{eq:prop-rearrange}
\text{$d_j$ is right continuous,}\quad  d_j(0)= \underset{\R^N}{\esssup}\: j, \quad \text{and}\quad |\{j \ge d_j(r)\}| \ge r \qquad \text{for every $r>0$.} 
\end{equation}
The first property is classical, and the second property is a consequence of the first. The third property in (\ref{eq:prop-rearrange}) is obvious if $d_j(r)=0$ since $|\{j \ge 0\}|= \R^N$. If $d_j(r)>0$, we have 
$$
d_j(r)= \inf \{c \ge 0 \::\: |\{j > c\}| \le r \} = \sup \{c \ge 0 \::\: |\{j > c\}| > r \}
$$
and therefore $|\{j \ge c\}|\ge |\{j > c\}| \ge r$ for every $c < d_j(r)$, whereas the property $j \in L^{1}(\R^{N}\setminus B_1(0))$ 
 implies that $|\{j \ge c\}|<\infty$ for every $c>0$. Since 
$$
\{j \ge d_j(r)\}= \underset{c < d_j(r)}{\bigcap}\{j \ge c\},
$$ 
we therefore conclude that $|\{j \ge d_j(r)\}| = \inf \limits_{c < d_j(r)} |\{j \ge c \}| \ge r.$

We now want to relate the decreasing rearrangement of $j$ to the killing measure associated with $j$ and a measurable set $\Omega \subset \R^N$, which is a function defined by   
\begin{equation}
  \label{eq:def-killing-measure}
\kappa_\Omega: \R^N \to [0,\infty], \qquad \kappa_\Omega(x)= \int_{\R^N \setminus \Omega}j(x-y)\,dy. 
\end{equation}

We have the following inequality.

\begin{lemma}
\label{sec:local-compactness-1-0}
For $\Omega \subset \R^N$ with $|\Omega| <\infty$ and $x \in \R^N$ we have 
\begin{equation}
  \label{eq:kappa-ineq}
\kappa_{\Omega}(x) \ge \kappa(|\Omega|),
\end{equation}
where $\kappa:[0,\infty) \to [0,\infty]$ is defined by 
\begin{equation}
\label{kappa-function}
\kappa(r):= \int_{\{ j < d_j(r)\}}\!\!j(z)\,dz \;+\; d_j(r) \Bigl(|\{j \ge d_j(r)\}|-r\Bigr)   \qquad \text{for $r \ge 0$.} 
\end{equation}
Moreover, if $j$ satisfies $(A2)$, then $\kappa(r) \to \infty$ as $r \to 0$.
\end{lemma}

\begin{proof}
The proof is somewhat similar to the proof of \cite[Proposition 3.3]{JW17}.
Without loss of generality, we may assume that $r:=|\Omega|>0$, since otherwise 
$$
\kappa_{\Omega}(x) = \int_{\R^N}j(x-y)\,dy=\int_{\R^N}j(z)\,dz= \kappa(0) \qquad \text{for every $x \in \R^N$.}
$$
>From $r>0$ we then deduce that $d:=d_j(r)<\infty$, and for every $x \in \R^N$ we have 
\begin{align*}
|\{j \ge d\} \setminus \Omega_x| -|\Omega_x \setminus \{j \ge d\}| &= |\{j \ge d\}|- |\{j \ge d\} \cap \Omega_x| -\Bigl(|\Omega_x|- |\{j \ge d\} \cap \Omega_x|\Bigr)\\
&= |\{j \ge d\}|-|\Omega_x|= |\{j \ge d\}|-r
\end{align*}
with $\Omega_x:= x+ \Omega$. Consequently, we have  
		\begin{align*}
\kappa_{\Omega}(x) &= \int_{\R^{N}\setminus \Omega_x} j(y)\ dy=\int_{\{j<d\}} j(y)\ dy +\int_{\{j \ge d\}\setminus \Omega_x}j(y)\ dy- \int_{\Omega_{x}\setminus \{j \ge d\}}j(y)\ dy\\
		&\ge \int_{\{j<d\}} j(y)\ dy + d \Bigl(|\{j \ge d\} \setminus \Omega_x| - 
		|\Omega_{x}\setminus \{j \ge d\} |\Bigr) = \int_{\{j<d\}} j(y)\ dy +d(|\{j \ge d\}|-r)\\
&= \kappa(r) \qquad \text{for $x \in \R^N$.}
 		\end{align*}
This shows (\ref{eq:kappa-ineq}). If in addition $j$ satisfies $(A2)$, then we have that $d_j(r) \to \infty$ as $r  \to 0$ and therefore 
$$
\int_{\{j < d_j(r)\}}\!\!j(z)\,dz \to \int_{\R^N} j(z)\,dz= \infty \qquad \text{as $r \to 0$.}
$$
Moreover, since $|\{j \ge d_j(r)\}| \ge r$ by (\ref{eq:prop-rearrange}), we have $\kappa(r) \ge \int_{\{j < d_j(r)\}}j(z)\,dz$ and therefore $\kappa(r) \to \infty$ as $r \to 0$.
\end{proof}

Next, for $t \ge 0$, we first consider the projection $P_t: L^2(\R^N) \to L^2(\R^N)$ onto the closed convex set 
$$
C_t:= \{u \in L^2(\R^N)\::\: |u| \le t \quad \text{a.e. on $\R^N$}\}.
$$
It is defined by 
\begin{equation}
  \label{eq:defi-projection}
[P_t u](x)= \left \{ 
  \begin{aligned}
  &u(x), &&\qquad |u(x)| \le t,\\
  &t, && \qquad u(x) > t,\\
  &-t, && \qquad u(x) < -t.  
  \end{aligned}
\right.
\end{equation}
Since 
$$
|[P_t u](x)-[P_t u](y)| \le |u(x)-u(y)| \qquad \text{for all $u \in \cD^j(\R^N)$, $x,y \in \R^N$,}
$$
if follows that $P_t(\cD^j(\R^N)) \subset \cD^j(\R^N)$, and that 
$$
\|P_t u\| \le \|u\| \qquad \text{for all $u \in \cD^j(\R^N)$, $t \ge 0$.}
$$
It is clear that the multiplication operator $R_\Omega: L^2(\R^N) \to L^2(\R^N)$ commutes with $P_t$ for all $t \ge 0$.

\begin{lemma}
\label{sec:local-compactness-1-0-0}
For $t>0$ and $u \in \cD^j(\R^N)$, we have 
$$
\|u-P_t u\|_{L^2(\R^N)} \le \Bigl(\:\inf_{x \in \Omega_{u,t}}
\kappa_{\Omega_{u,t}}(x)\Bigr)^{-\frac{1}{2}} \|u\|   \qquad \text{with $\Omega_{u,t}:= \{x \in \R^N \::\: |u(x)| > t\}$.}
$$
\end{lemma}

\begin{proof}
Since $|u(x)|-|u(y)| \ge |u(x)|-t \ge 0$ for $x \in \Omega_{u,t}$, $y \in \R^N \setminus \Omega_{u,t}$, 
we have
\begin{align*}
\|u\|^2 &\ge  \int_{\R^N \setminus \Omega_{u,t}} \int_{\Omega_{u,t}}(|u(x)|-|u(y)|)^2j(x-y)\,dxdy\\
&\ge \int_{\Omega_{u,t}} \bigl(|u(x)|-t \bigr)^2 \int_{\R^N \setminus \Omega_{u,t}} j(x-y)\,dy dx = 
\int_{\Omega_{u,t}} \bigl(|u(x)|-t \bigr)^2 \kappa_{\Omega_{u,t}}(x)\,dx\\
&\ge \Bigl(\:\inf_{x \in \Omega_{u,t}}
\kappa_{\Omega_{u,t}}(x)\Bigr) \int_{\Omega_{u,t}} \bigl(|u(x)|-t \bigr)^2\,dx= \Bigl(\: \inf_{x \in \Omega_{u,t}}
\kappa_{\Omega_{u,t}}(x)\Bigr)\|u-P_t u\|_{L^2(\R^N)}^2,
\end{align*}
as claimed.
\end{proof}

\begin{cor}
\label{sec:local-compactness-1}
For $t>0$ and $u \in \cD^j(\R^N)$, we have 
$$
\|u-P_t u\|_{L^2(\R^N)} \le \bigl[\kappa(|\Omega_{u,t}|)\bigr]^{-\frac{1}{2}} \|u\|
$$
with $\Omega_{u,t}$ as in Lemma~\ref{sec:local-compactness-1-0} and the function $\kappa$ defined in (\ref{kappa-function}).
\end{cor}

\begin{proof}
This is a direct consequence of Lemmas~\ref{sec:local-compactness-1-0} and \ref{sec:local-compactness-1-0-0}. Here we note that $|\Omega_{u,t}|< \infty$ for $t>0$ since $u \in L^2(\R^N)$.  
\end{proof}

We are now ready to complete the proof of Theorem~\ref{thm-cor-sec:local-compactness-5}, which is a direct consequence of the following result. 

\begin{thm}
\label{sec:local-compactness-4}
Let $j: \R^N \to [0,\infty]$ satisfy $(A1)$ and $(A2)$, and let $\Omega \subset \R^N$ be a measurable subset with $|\Omega|<\infty$. Then $R_\Omega$ is compact as an operator $\cD^j(\R^N) \to L^2(\R^N)$.
\end{thm}

\begin{proof}
Let $M \subset \cD^j(\R^N)$ be a bounded set with $C:= \sup \limits_{u \in M}\|u\|$. To show that $R_\Omega(M) \subset L^2(\R^N)$ is relatively compact, we let $\eps>0$ and choose $t>0$ sufficiently large to guarantee that 
$$
\bigl(\kappa(|\Omega_{u,t}|)\bigr)^{-\frac{1}{2}}< \frac{\eps }{2C} \qquad\text{ for all $u\in M$.}
$$ 
This is possible since $\kappa(r) \to \infty$ as $r \to 0$ by Lemma~\ref{sec:local-compactness-1-0} and since 
$$
|\Omega_{u,t}|\leq \frac{\|u\|_{L^2(\R^N)}^2}{t^2}\leq \frac{C^2}{t^2}\qquad \text{for every $u\in M$, $t>0$.}
$$
Moreover, by the inner regularity of Lebesgue measure and since $|\Omega| < \infty$, we may choose a compact set $K \subset \Omega$ with 
$$
t |\Omega \setminus K|^{\frac{1}{2}} \le \frac{\eps}{2}.
$$
By Corollary, \ref{sec:local-compactness-1}, we then have 
\begin{align*}
&\|R_\Omega u - R_K u\|_{L^2(\R^N)} = \|R_{\Omega \setminus K}u \|_{L^2(\R^N)} \le 
\|R_{\Omega \setminus K}P_t u \|_{L^2(\R^N)}+ \|R_{\Omega \setminus K}(u-P_t u) \|_{L^2(\R^N)}\\
&\le  |\Omega \setminus K|^{\frac{1}{2}}\|P_t u\|_{L^\infty(\R^N)} + \|u-P_t u \|_{L^2(\R^N)}
\le t |\Omega \setminus K|^{\frac{1}{2}} + \bigl(\kappa(|\Omega_{u,t}|)\bigr)^{-\frac{1}{2}}\|u\| \le \eps\qquad \text{for $u \in M$.} 
\end{align*}
Hence $R_\Omega(M)$ is contained in the $\eps$-neighborhood of the set $R_K(M)$ in $L^2(\R^N)$. Since $\eps>0$ was chosen arbitrarily and $R_K(M)$ is compact by Theorem~\ref{sec:local-compactness-2}, it follows that $R_\Omega(M)$ is totally bounded in $L^2(\R^N)$. Hence it is relatively compact in $L^2(\R^N)$, as claimed.
\end{proof}

As mentioned above, Theorem \ref{thm-cor-sec:local-compactness-5} is an immediate consequence of Theorem \ref{sec:local-compactness-2}. We finally complete the 

\begin{proof}[Proof of Theorem \ref{sec:vanishing-nonvanishing}]
Let $C:= \sup \limits_{n \in \N}\|u_n\|$, and suppose that alternative (i) does not hold. Then there exists $\eps,\delta>0$ and a subsequence -- still denoted by $(u_n)_n$ --  with the property that
$$
\int_{|u_n| \ge 2 \eps}u_n^2\,dx \ge \delta \qquad \text{for all $n \in \N$.}
$$
Since $|u_n|^2 \le 4(|u_n|-\eps)^2$ on the set $\{|u_n| \ge 2\eps\}$, we deduce that  
$$
\delta \le  4 \int_{|u_n| \ge \eps}(|u_n| -\eps)^2\,dx = 4\|u_n-P_\eps u_n\|_{L^2(\R^N)}^2
$$
where $P_\eps: L^2(\R^N) \to L^2(\R^N)$ is the projection on the convex set $\{u \in L^2(\R^N)\::\: |u| \le \eps\}$ as defined in (\ref{eq:defi-projection}). Let $\Omega_n:= \{x \in \R^N \::\: |u_n(x)| > \eps\}$ and $\kappa_n(x):= \kappa_{\Omega_n}(x)$ for $n \in \N$, $x \in \Omega_n$. By Lemma~\ref{sec:local-compactness-1-0-0}, we then have, for all $n \in \N$, 
$$
\frac{\delta}{4} \le \|u-P_\eps u\|_{L^2(\R^N)}^2 \le \frac{\|u\|^2}{k_n^*} \le 
\frac{C^2}{k_n^*}\qquad  \text{with}\quad k_n^*:= \inf_{x \in \Omega_n}\kappa_n(x). 
$$
Hence the sequence $(k_n^*)_n$ remains bounded. Let $x_n \in \Omega_n$ be chosen such that 
\begin{equation}
  \label{eq:kappa-n-x-n-est}
\kappa_n(x_n) \le k_n^* +1, 
\end{equation}
and let $v_n:= x_n * u_n$. Then we have 
\begin{equation}
  \label{eq:kappa-n-lower-bound}
\kappa_n(x) +1 \ge \int_{\R^N \setminus \Omega_n} j(x_n-y)\,dy = \int_{\{|v_n| < \eps\}}j(y)\,dx = \int_{\R^N} j_n(y)\,dy \qquad \text{for all $x\in \R^N$, $n \in \N$}
\end{equation}
with $j_n:= j \,1_{\{|v_n| < \eps\}}: \R^N \to [0,\infty]$. Since $(v_n)_n$ is bounded in $\cD^j(\R^N)$, we may, by Theorem \ref{sec:local-compactness-2}, pass to a subsequence with  
$$
v_n \weak u\quad \text{in $\cD^j(\R^N)$},\qquad 
v_n \to u \quad \text{in $L^2_{loc}(\R^N)$},\qquad \text{and}\qquad
v_n \to u\quad \text{a.e. in $\R^N$.}
$$
Suppose by contradiction that $u=0$. Then we have $j_n \to j$ a.e. in $\R^N$. Therefore (\ref{eq:kappa-n-lower-bound}) and Fatou's Lemma imply that 
$$
\liminf_{n \to \infty} \kappa_n(x_n) \ge \int_{\R^N} j(y)\,dy = \infty.
$$
This is a contradiction, as the sequence $\kappa_n(x_n)$ is bounded by (\ref{eq:kappa-n-x-n-est}) and since $(k_n^\ast)_n$ is bounded. It follows that $u \not = 0$, as claimed.
\end{proof}

\section{Application to a maximization problem}
\label{sec:appl-maxim-probl}

The present section is devoted to the proof of Theorem~\ref{max-value-attained}.
Throughout this section, we assume that 
$j: \R^N \to [0,\infty]$ satisfies $(A1)$ and $(A2)$. 

We first need the following preliminary lemma.

\begin{lemma}
\label{density-compact-support}
The subspace of functions in $\cD^j(\R^N)$ with bounded support is dense in $\cD^j(\R^N)$. 
\end{lemma}

\begin{proof}
For $R>0$, let $\phi_R: \R^N \to \R$ be Lipschitz functions with $0 \le \phi_R \le 1$, $\phi_R \equiv 1$ on $B_R$, $\phi_R \equiv 0$ on $\R^N \setminus B_{2R}$ and 
$$
|\phi_R(x) -\phi_R(y)| \le \frac{|x-y|}{R} \qquad \text{for $x,y \in \R^N$.}
$$
Moreover, let $\psi_R:= 1 -\phi_R$ for $R>0$, and let $u \in \cD^j(\R^N)$. We claim that 
\begin{equation}
  \label{eq:density-proof-necessary}
\cE_j(u- u \phi_R,u- u \phi_R)= \cE_j(u \psi_R,u \psi_R) \to 0 \qquad \text{as $R \to \infty$.}
\end{equation}
Indeed, we have 
\begin{align*}
&\cE_j(u \psi_R,u \psi_R)= \frac{1}{2}\int_{\R^N}\int_{\R^N} [u(x)\psi_R(x)-u(y)\psi_R(y)]^2j(x-y)\ dydx\\
        &= \frac{1}{2}\int_{\R^N}\int_{\R^N} \bigl[u(x)\bigl(\psi_R(x)-\psi_R(y)\bigr)+ \psi_R(y)\bigl(u(x)-u(y)\bigr)\bigr]^2j(x-y)\ dydx\\
	&\leq \int_{\R^N}u(x)^2 K_R(x) dx+ \int_{\R^N}\int_{\R^N} \psi_R(y)^2[(u(x)-u(y)]^2j(x-y)\ dydx  
\end{align*}
with 
$$
K_R(x)= \int_{\R^N}[\psi_R(x)-\psi_R(y)]^2j(x-y)\ dy= \int_{\R^N}[\phi_R(x)-\phi_R(y)]^2j(x-y)\ dy \quad \text{for $x \in \R^N$.}
$$
Since $|\psi_R| \le 1$ and $\psi_R \to 0$ pointwise on $\R^N$ as $R \to \infty$, Lebesgue's theorem implies that 
$$
\int_{\R^N}\int_{\R^N} \psi_R(y)^2[(u(x)-u(y)]^2j(x-y)\ dydx \to 0 \qquad \text{as $R \to \infty$.}
$$
Moreover, since 
$$
|\phi_R(x)-\phi_R(y)| \le \min \{1, \frac{|x-y|}{R}\} \le \min \{1,|x-y|\}
\qquad \text{for $x,y \in \R^N$, $R \ge 1$,}
$$
it follows from $(A1)$ and Lebesgue's theorem that 
$$
|K_R(x)| \le C_j:= \int_{\R^N} \min \{1,|z|^2\}j(z)\,dz < \infty \qquad \text{for $x \in \R^N$, $R \ge 1$}
$$
and 
$$
K_R(x) \to 0 \qquad \text{as $R \to \infty\quad $ for every $x \in \R^N$.} 
$$
Applying Lebesgue's theorem again, we find that
$$
\int_{\R^N}u(x)^2 K_R(x) dx \to 0 \qquad \text{as $R \to \infty$.}
$$
We thus obtain (\ref{eq:density-proof-necessary}). Since also 
$$
\|u- u \phi_R\|_{L^2(\R^N)}^2 = \|u \psi_R\|_{L^2(\R^N)}^2 \to 0 \qquad \text{as $R \to \infty$,}
$$
we conclude that 
$$
\|u- u \phi_R\| \to 0 \qquad \text{as $R \to \infty$.}
$$
This shows that functions with compact support are dense in $\cD^j(\R^N)$. 
 \end{proof}

 \begin{remark}
 \label{density-C-infty-c}
Combining Lemma~\ref{density-compact-support} with \cite[Proposition 4.1]{JW17}), we deduce that the space $C_c^\infty(\R^N)$ is dense in $\cD^j(\R^N)$. However, Lemma~\ref{density-compact-support} is sufficient for our purposes here. 
 \end{remark}

Next, let $F \in C(\R)$ satisfy properties (F1)--(F3).  As in the introduction, we consider the integral functional 
$$
\Phi: L^2(\R^N) \to \R, \qquad \Phi(u)= \int_{\R^N}F(u)\,dx,
$$
and the maximization problem associated with 
$$
m_{F,\R^N}:= \sup_{u \in S } \Phi(u),
$$
where $S(\R^N):= \{u \in \cD^j(\R^N) \::\: \|u\| =1\}$ is the unit sphere in $\cD^j(\R^N)$. For simplicity, we write $m_F$ and $S$ in place of $m_{F,\R^N}$ and $S(\R^N)$ in the following. 

Clearly, we have $m_{F}>0$ if and only if $F\not\equiv 0$. Moreover, by (F1) we have
$$
\Phi(u)=  \int_{\R^N}F(u)\,dx \le c_\infty \|u\|_{L^2(\R^N)}^2 \le c_\infty
\qquad \text{for $u \in S$}
$$
and therefore $m_{F}\le c_\infty<\infty$. From the same bound, it also follows by a classical argument that $\Phi$ is continuous on $L^2(\R^N)$.  We also need the following lemma.

\begin{lemma}
\label{sec:comp-stat-2}
If $w \in \cD^j(\R^N)$ satisfies $\|w\| \le 1$, then we have $\Phi(w) \le m_F \|w\|^2$. Moreover, if $w \not = 0$ and 
$\phi(w)= m_F \|w\|^2$, then 
$\tilde w = \frac{w}{\|w\|}$ is a maximizer of $\Phi$ on $S$.  
\end{lemma}

\begin{proof}
The claim is obvious if $w = 0$. Hence we assume that $w \not = 0$, and we let $t= \frac{1}{\|w\|}\ge 1$. Then we have $\|t w\|=1$ and therefore, by the definition of $m_F$ and (F3),  
\begin{equation}
  \label{eq:equality-case}
m_F \ge \Phi(t w)= \int_{\R^N} F(t w)\,dx \ge t^2 \int_{R^N} F(w)\,dx = t^2\Phi(w)= \frac{\Phi(w)}{\|w\|^2}.
\end{equation}
Hence $\Phi(w) \le m_F \|w\|^2$. Moreover, if $w \not = 0$ and 
$\phi(w)= m_F \|w\|^2$, then equality holds in (\ref{eq:equality-case}) and therefore 
$\Phi(t w)=m_F$. Hence $tw = \frac{w}{\|w\|}$ is a maximizer of $\Phi$ on $S$. 
\end{proof}

We may now complete the 

\begin{proof}[Proof of Theorem \ref{max-value-attained}]
The claim is obvious if $m_F=0$, so we assume that $m_F>0$ in the following. Let $(u_n)_n \subset S$ be a maximizing sequence, i.e., we have $\|u_n\|=1$ for all $n \in \N$ and $\Phi(u_n) \to m_F$ as $n \to \infty$. If 
\begin{equation}
  \label{eq:first-alt-proof-max}
\lim \limits_{n \to \infty} \int_{|u_n| \ge \eps}u_n^2\,dx = 0\qquad \text{for every $\eps>0$,}
\end{equation}
we deduce from (F1) that 
$$
\Phi(u_n)\le c_\eps \int_{|u_n| < \eps}|u_n|^2\,dx + {c_\infty \,} \int_{|u_n| \ge \eps}u_n^2\,dx  \le c_\eps \|u_n\|_{L^2}^2 + o(1)
\le c_\eps + o(1)\quad \text{as $n \to \infty$}
$$
for every $\eps>0$, where
\begin{equation}
  \label{eq:def-c-eps}
c_\eps:= \sup_{0< |t| \le \eps} \frac{|F(t)|}{t^2} \to 0 \qquad \text{as $\eps \to 0$}
\end{equation}
by (F2). Hence $\Phi(u_n) \to 0$ as $n \to \infty$; a contradiction. 
We thus conclude that (\ref{eq:first-alt-proof-max}) does not hold, so by Theorem~\ref{sec:vanishing-nonvanishing} there exists a sequence $(x_n)_n \subset \R^N$ such that, after passing to a subsequence,  
$$
x_n * u_n \weakto \psi \not = 0 \qquad \text{in $\cD^j(\R^N)$.}
$$
By translation invariance of $\Phi$ and $S$, we may replace the sequence $(u_n)_n$ by $(x_n*u_n)_n$, which gives that 
\begin{equation}
  \label{eq:translated-weak}
u_n \weakto \psi \not = 0 \qquad \text{in $\cD^j(\R^N)$.}
\end{equation}
Consequently, since $\|u_n\|=1$ for all $n$, we have 
$$
0 < \|\psi\| \le  1 \qquad \text{and}\qquad \|u_n-\psi\|^2= 1-\|\psi\|^2 + o(1) \qquad \text{as $n \to \infty$.} 
$$
Passing to a subsequence, we may thus assume that $\|u_n- \psi\|< 1$ for all $n \in \N$. 
We claim that
\begin{equation}
  \label{eq:claim-psi-phi}
\Phi(\psi)\ge m_F \|\psi\|^2.  
\end{equation}
Suppose by contradiction that 
$$
\Phi(\psi) < m_F \|\psi\|^2-\delta \qquad \text{for some $\delta >0$.}
$$
Making $\delta$ smaller if necessary, we may assume that $\delta < \min \{1,\|\psi\|^2\}$. Since functions with bounded support are dense in $\cD^j(\R^N)$ by Lemma~\ref{density-compact-support} and $\Phi$ is continuous on $\cD^j(\R^N) \subset L^2(\Omega)$, there exists $\phi \in \cD^j(\R^N)$ with bounded support and such that 
\begin{align}
&\Phi(\psi)-\frac{\delta}{2} < \Phi(\phi)< m_F \|\phi\|^2-\delta, \label{approx-compact-supp1}\\
&\|\phi-\psi\|<\delta_1:= \frac{\delta}{16\max \{1,m_F\}}\leq 1, \label{approx-compact-supp2}\\
&\text{and}\qquad \|\phi\|\le 2.\label{approx-compact-supp4}
\end{align}

We set $w_n:= u_n- \phi$ for $n \in \N$. Since $\|u_n-\psi\|<1$ for all $n \in \N$ and $\|\phi\| \le 2$, we have  
\begin{equation}
  \label{eq:tau-ineq}
(\|u_n-\psi\|+\tau)^2 \le \|u_n-\psi\|^2 + 3 \tau \quad \text{and}\quad (\|\phi\|-\tau)^2 \ge \|\phi\|^2 - 4\tau 
\end{equation}
for $\tau \in (0,1)$ and $n \in \N$. Consequently, 
\begin{align}
\|w_n\|^2 &\le \Bigl(\|u_n-\psi\|+\|\psi-\phi\|\Bigr)^2 \le \Bigl(\|u_n-\psi\|+\delta_1\Bigr)^2 \nonumber\\
&\le \|u_n-\psi\|^2 +3 \delta_1 = 1-\|\psi\|^2 +3 \delta_1 + o(1) \quad\text{as $n\to\infty$.}\label{w-n-norm-ineq}
\end{align}
Since $3 \delta_1 < \delta < \|\psi\|^2$,  we may pass to a subsequence such that $\|w_n\| \le 1$ for all $n \in \N$, which by Lemma~\ref{sec:comp-stat-2} implies that 
\begin{equation}
\label{w-n-ineq}
\Phi(w_n) \le m_F \|w_n\|^2 \qquad \text{for all $n \in \N$.}
\end{equation}
It also follows from (\ref{eq:tau-ineq}) and (\ref{w-n-norm-ineq}) that
\begin{align}
\|w_n\|^2 & \le 1-\|\psi\|^2 +3 \delta_1 + o(1) \le 1 -(\|\phi\|-\delta_1)^2
+3 \delta_1 +o(1)\notag\\
&\le 1 -\|\phi\|^2 + 7 \delta_1 +o(1) \quad\text{as $n\to\infty$.}\label{w-n-ineq2}
\end{align}
Next, let $M:= B_R(0)$, 
where $R>0$ is chosen sufficiently large to guarantee that $\supp\, \phi \subset M$.
We write 
$$
\Phi= \Phi_1+ \Phi_2
$$
with 
$$
\Phi_1, \Phi_2: L^2(\R^N) \to \R, \quad \Phi_1(u)=\Phi(u 1_M)= \int_{M}F(u)\,dx,\quad \Phi_2(u)=\Phi_2(u 1_{\R^N \setminus M})= \int_{\R^N \setminus M}F(u)\,dx.
$$
Since $F$ is nonnegative, we have 
$$
\Phi_1(v) \le \Phi(v) \quad \text{and}\quad \Phi_2(v) \le \Phi(v) \qquad \text{for all $v \in L^2(\R^N)$.}
$$
Moreover, since $u_n 1_M \to \psi 1_M$ in $L^2(\R^N)$ by Theorem \ref{sec:local-compactness-2},  
$$
\Phi_1(u_n) = \Phi_1(\psi) + o(1) \le  \Phi(\psi) + o(1) \le \Phi(\phi) + \frac{\delta}{2}+ o(1)\quad\text{as $n\to\infty$.}
$$
Here we used (\ref{approx-compact-supp1}) in the last inequality. Since $u_n= w_n$ on $\R^N \setminus M$, we also have 
$$
\Phi_2(u_n)= \Phi_2(w_n) \le \Phi(w_n)
$$
Combining these estimates, we find, by \eqref{approx-compact-supp1}, \eqref{w-n-ineq}, and \eqref{w-n-ineq2},
\begin{align*}
\Phi(u_n) &= \Phi_1(u_n)+ \Phi_2(u_n) \le \Bigl(\Phi(\phi) + \frac{\delta}{2}+ o(1)\Bigr) + \Phi(w_n) \le \Bigl( m_F \|\phi\|^2 - \frac{\delta}{2}+ o(1)\Bigr)+ m_F \|w_n\|^2\\ 
&\le \Bigl( m_F \|\phi\|^2 -\frac{\delta}{2}+ o(1)\Bigr)+ m_F \Bigl(1 -\|\phi\|^2 + 7 \delta_1 +o(1)\Bigr)\\
&= m_F  -\frac{\delta}{2} + 7 m_F \delta_1 + o(1) \le m_F- \frac{\delta}{16}+o(1)\quad\text{as $n\to\infty$}
\end{align*}
since $7 \delta_1 m_F \le \frac{7\delta}{16}$ by \eqref{approx-compact-supp2}. We conclude that 
$$
m_F = \lim_{n \to \infty} \Phi(u_n) \le m_F -\frac{\delta}{16}.
$$
This is a contradiction, and hence (\ref{eq:claim-psi-phi}) holds. Since $\|\psi\| \le 1$, Lemma~\ref{sec:comp-stat-2} now implies that $u:= \frac{\psi}{\|\psi\|}$ is a maximizer of $\Phi$ on $S$.
\end{proof}

\section{The Poincar\'e inequality}
\label{sec:poincare-inequality}

For the proof of Proposition \ref{poincare} we need the following Lemma, which is a simple variant of \cite[Lemma 10]{DK11}. We include the proof for the convenience of the reader. Recall that, for a nonnegative even function $q \in L^1(\R^N)$, 
the corresponding bilinear form $\cE_q(u,v)$ is given by (\ref{bilinear}) with $q$ in place of $j$.

	\begin{lemma}\label{iterationlemma}
		Let $q\in L^1(\R^N)$ be a nonnegative even function. Then for all measurable functions $u:\R^N\to \R$ we have
			\[
			\cE_{q\ast q}(u,u)\leq 4\|q\|_{L^1(\R^N)} \cE_{q}(u,u).
			\]
	\end{lemma}
	\begin{proof}
		Let $u$ be as stated and denote $g(x,y)=(u(x)-u(y))^2$ for $x,y\in \R^N$. Note that we have
		\[
		0\leq g(x,y)=g(y,x)\leq 2g(x,z)+2g(y,z)\qquad\text{ for all $x,y,z\in \R^N$.}
		\]
		By Fubini's theorem we have
		\begin{align*}
		\int_{\R^N}\int_{\R^N}&g(x,y) (q\ast q)(x-y)\ dxdy=	\int_{\R^N}\int_{\R^N} \int_{\R^N} g(x,y) q(x-z)q(y-z)\ dzdxdy\\
		&\leq	2\int_{\R^N}\int_{\R^N} \int_{\R^N} [g(x,z)+g(y,z)] q(x-z)q(y-z)\ dzdxdy\\
		&= 4\int_{\R^N}\int_{\R^N}g(x,z)q(x-z) \int_{\R^{N}}q(y-z)\ dydzdx= 4\|q\|_{L^1(\R^N)}\int_{\R^N}\int_{\R^N}g(x,z)q(x-z)\ dxdz.
		\end{align*}
	\end{proof}

We may now complete the 

\begin{proof}[Proof of Proposition \ref{poincare}]
The first part of the argument follows exactly the lines of \cite[Proof of Lemma 2.7]{FKV13}. By assumption (A1), we have 
$q:=\min\{1,j\} \in L^1(\R^N) \cap L^\infty(\R^N)$, whereas
\begin{equation}
  \label{eq:est-j-q}
	\cE_j(u,u)\geq \cE_q(u,u)=\frac{1}{2}\int_{\R^N}\int_{\R^N}(u(x)-u(y))^2q(x-y)\ dxdy \quad \text{ for all $u\in \cD^j(\Omega)$.}
\end{equation}
For $m \in \N$ we consider the $2^m$-fold convolution
	\begin{equation}\label{def:q}
	q_m := \underbrace{q\ast\ldots\ast q}_{\text{$2^m$-times}}  \qquad \text{in $L^1(\R^N)$.} 
	\end{equation}
By Young's inequality, we have $\|q_m\|_{L^1(\R^N)}\le \|q\|_{L^1(\R^N)}^{2^m}$. Moreover, $q_m$ is bounded and continuous for $m \in \N$, and 
$$
q_1(0)= \int_{\R^N}q^2(y)\ dy>0
$$
since $j \not \equiv 0$ and therefore $q \not \equiv 0$. By continuity, we have 
$\inf_{B_{\delta}(0)}q_1>0$ for some $\delta>0$, and then standard properties of convolution yield that $\inf_{B_{2^m\delta}(0)}q_m>0$. 
For $a>0$, we put $\Omega_a:= (-a,a)\times \R^{N-1} \subset \R^N$ in the following. Without loss of generality, we may assume that $\Omega= \Omega_a$ for some fixed $a>0$. We then fix $m \in \N$ with $2^m\delta >2 a$. Then $|B_{2^m\delta}(0)\setminus \Omega_{2a}|>0$ and therefore 
$$
	\int_{\R^N \setminus \Omega_a} q_{m}(x-y)\ dy= \int_{\{z_1<-x_1-a\}\cup \{z_1>-x_1+a\}}\!\!\!\! q_{m}(z)\ dz\geq \int_{\R^N \setminus \Omega_{2a}} q_{m}(z)\ dz=:C_{a,1}>0 \quad \text{for $x \in \Omega_a$.}
$$
	Thus, by Lemma \ref{iterationlemma} and (\ref{eq:est-j-q}), we have that	\begin{align*}
	C_{a,1}\int_{\Omega_a} u^2(x) \ dx  &\leq \int_{\Omega_a}u^2(x) \int_{\R^N \setminus \Omega_a} q_{m}(x-y)\ dy \le \cE_{q_m}(u,u) \le \Bigl(4^{2m}\prod_{k=0}^{m-1}\|q_{k}\|_{L^1(\R^N)}\Bigr) \cE_q(u,u)\\
&\le \Bigl(4^{2m}\prod_{k=0}^{m-1}\|q_{k}\|_{L^1(\R^N)}\Bigr) \cE_j(u,u) \qquad\text{for every  $u\in \cD^j(\Omega_a)$.}
	\end{align*}
Consequently, (\ref{eq:poincare-first-claim}) follows with with $C_a=4^{-2m}C_{a,1}\Bigl(\,\prod \limits_{k=0}^{m-1}\|q_{k}\|_{L^1(\R^N)}\Bigr)^{-1}$.
	To see that $\liminf \limits_{a \to 0^+} \, C_a \ge \int_{\R^N}j(z)\ dz$, it is enough to note that, similarly as above, 
$$
\cE_j(u,u)\geq \int_{\Omega_a}u^2(x)\int_{\R^N \setminus \Omega_a}j(x-y)\ dydx \geq \tilde C_a \int_{\Omega_a}u^2(x)\,dx \quad \text{for $u\in \cD^j(\Omega_a)$}$$
with $\tilde C_a:=   
\int_{\R^N \setminus \Omega_{2a}}j(z) dz$, whereas $\lim\limits_{a\to 0^+} \tilde C_a=\int_{\R^N}j(z)\ dz$ by monotone convergence.
\end{proof}

\bibliographystyle{amsplain}

\end{document}